\documentclass{amsart}%[a4paper]
\usepackage{graphicx}
\usepackage{amsmath}
\usepackage{amssymb}
\usepackage[all]{xy}

\newcommand{\Hom}{\operatorname{Hom}\nolimits}
\renewcommand{\Im}{\operatorname{Im}\nolimits}

\newcommand{\pd}{\operatorname{pd}\nolimits}

\newcommand{\depth}{\operatorname{depth}\nolimits}
\newcommand{\length}{\operatorname{length}\nolimits}
\newcommand{\Ann}{\operatorname{Ann}\nolimits}

\newcommand{\Tor}{\operatorname{Tor}\nolimits}
\newcommand{\Ext}{\operatorname{Ext}\nolimits}
\newcommand{\Tatetor}{\operatorname{\widehat{Tor}}\nolimits}
\newcommand{\Tateext}{\operatorname{\widehat{Ext}}\nolimits}

\newcommand{\Id}{\operatorname{Id}\nolimits}
\renewcommand{\H}{\operatorname{H}\nolimits}

\newcommand{\m}{\operatorname{\mathfrak{m}}\nolimits}
\newcommand{\n}{\operatorname{\mathfrak{n}}\nolimits}

\newcommand{\cx}{\operatorname{cx}\nolimits}

\newcommand{\cidim}{\operatorname{CI-dim}\nolimits}
\newcommand{\gdim}{\operatorname{G-dim}\nolimits}

\newtheorem{theorem}{Theorem}[section]

\newtheorem{corollary}[theorem]{Corollary}

\newtheorem{lemma}[theorem]{Lemma}

\theoremstyle{definition}
\newtheorem{definition}[theorem]{Definition}
\theoremstyle{definition}

\theoremstyle{definition}

\theoremstyle{definition}
\newtheorem{example}[theorem]{Example}
\theoremstyle{definition}

\theoremstyle{definition}

\theoremstyle{remark}
\newtheorem*{remark}{Remark}
\theoremstyle{definition}

\theoremstyle{definition}

\begin{document}

\title[Vanishing of homology]{On the vanishing of homology for modules of finite complete intersection
dimension}

\author{Petter Andreas Bergh \& David A.\ Jorgensen}

\address{Petter Andreas Bergh \\ Institutt for matematiske fag \\
  NTNU \\ N-7491 Trondheim \\ Norway}
\email{bergh@math.ntnu.no}

\address{David A.\ Jorgensen \\ Department of mathematics \\ University
of Texas at Arlington \\ Arlington \\ TX 76019 \\ USA}
\email{djorgens@uta.edu}

\begin{abstract} We prove rigidity type results on the
vanishing of stable $\Ext$ and $\Tor$ for modules of finite
complete intersection dimension, results which generalize and
improve upon known results. We also introduce a notion of
pre-rigidity, which generalizes phenomena for modules of finite
complete intersection dimension and complexity one. Using this
concept, we prove results on length and vanishing of homology
modules.
\end{abstract}



\maketitle

\section{Introduction}

The notion of rigidity of $\Tor$ was introduced by Auslander
\cite{Auslander} in order to study torsion in tensor products, and
the zerodivisor conjecture, for finitely generated modules over a
commutative local ring. The general idea of rigidity of Tor for
modules $M$ and $N$ over a ring $A$ is that the vanishing of
$\Tor^A_i(M,N)$ for some $i$ implies the vanishing of
$\Tor^A_j(M,N)$ for $j$'s different from $i$. Ever since its
introduction by Auslander, rigidity of $\Tor$ has been a central
topic in the theory of modules over commutative rings (see, for
example, \cite{PeskineSzpiro}, \cite{Hochster}, and
\cite{Heitmann}).

Rigidity of $\Tor$ for finitely generated modules over unramified
regular local rings was resolved by Auslander himself, and the
ramified case was settled by Lichtenbaum \cite{Lichtenbaum}.
The next natural class
of rings over which to study rigidity is that of complete
intersections, and this was done in \cite{HunekeWiegand1},
\cite{HunekeWiegand2}, \cite{Jorgensen}, and more recently,
\cite{Dao1} and \cite{Dao2}. Subsequent to the
notion of complete intersection dimension, defined in
\cite{AvramovGasharovPeeva}, there has been a study of rigidity
of Tor and Ext for modules of finite complete intersection dimension,
for example \cite{ArayaYoshino}, \cite{Jorgensen2},
\cite{AvramovBuchweitz}, and \cite{Bergh2}.

In this paper, we prove new rigidity results for $\Ext$ and $\Tor$
which generalize or improve upon many of the results in the above
citations.  We do so in the context of stable (co)homology. We
show in Section 3 that the vanishing of $c$ ($c$ being the complexity of one
of the modules) equally spaced stable
$\Ext$ or $\Tor$ implies the vanishing of infinitely many of the
remaining (co)homology modules.  We also show that if $\dim R+2$
consecutive stable $\Ext$ or $\Tor$ vanish infinitely often for negative
or positive indices, respectively, then
all the stable $\Ext$ or $\Tor$ must vanish.

In Section 4 we introduce a notion we call {\em pre-rigidity}, and
show that it generalizes the vanishing phenomena of modules of
finite complete intersection dimension and complexity one.  We
also show that it gives a formula for length which recovers known
results for Betti numbers of modules over rings having an embedded
deformation.

In Section 2 we give preliminaries on complete intersection dimension, complexity, and stable
(co)homology.

\section{Finite complete intersection dimension}

Throughout this section, we fix a local (meaning commutative
Noetherian local) ring ($A, \m, k$), together with a finitely
generated $A$-module $M$. Given a minimal free resolution
$$\cdots \to F_2 \to F_1 \to F_0
\to M \to 0$$ of $M$, we denote the rank of the free module $F_n$
by $\beta_n(M)$. This integer, the $n$th \emph{Betti number} of
$M$, is well-defined for all $n$, since minimal free resolutions
over local rings are unique up to isomorphisms. The
\emph{complexity} of $M$, denoted $\cx M$, is defined as
$$\cx M \stackrel{\text{def}}{=} \inf \{ t \in \mathbb{N} \cup
\{ 0 \} \mid \exists a \in \mathbb{R} \text{ such that } \beta_n (M)
\le an^{t-1} \text{ for all } n \gg 0 \}.$$ The complexity of a finitely
generated module over a local ring is not always finite; by a
theorem of Gulliksen (cf.\ \cite{Gulliksen}), the local rings over
which all finitely generated modules have finite complexity are
precisely the complete intersections.

In \cite{AvramovGasharovPeeva}, Avramov, Gasharov and Peeva
defined and studied a class of modules behaving homologically as
modules over complete intersections. Recall that a
\emph{quasi-deformation} of $A$ is a
diagram $A \to R \leftarrow Q$ of local
homomorphisms, in which $A \to R$ is faithfully flat, and $R
\leftarrow Q$ is surjective with kernel generated by a regular
sequence. The module $M$ has \emph{finite complete intersection
dimension} if there exists such a quasi-deformation for which
$\pd_Q (R \otimes_A M)$ is finite. The complete intersection
dimension of $M$, denoted $\cidim M$, is the infimum of all $\pd_Q
(R \otimes_A M) - \pd_Q R$, the infimum taken over all
quasi-deformations $A \to R \leftarrow Q$ of $A$. In the rest of
the paper, we write ``CI-dimension" instead of ``complete
intersection dimension".

By \cite[Theorem 5.3]{AvramovGasharovPeeva}, every module of
finite CI-dimension has finite complexity. Moreover, as we shall
see in the next section, such a module also has reducible
complexity in the sense of \cite{Bergh1}. This reflects the fact
that modules of finite CI-dimension behave homologically as
modules over complete intersections. Since complete intersection
rings are Gorenstein, modules of finite CI-dimension also behave,
in some sense, as modules over Gorenstein rings. In order to make
this precise, we recall the following, denoting the $A$-module
$\Hom_A(M,A)$ by $M^*$. We say that $M$ is of \emph{Gorenstein
dimension zero}, denoted $\gdim M =0$, if it is reflexive (i.e.\
the canonical homomorphism $M \to M^{**}$ is bijective) and
$\Ext_A^n(M,A)= \Ext_A^n(M^*,A)=0$ for $n>0$. The \emph{Gorenstein
dimension} of $M$, denoted $\gdim M$, is the infimum of the
numbers $n$, for which there exists an exact sequence
$$0 \to G_n \to \cdots \to G_0 \to M \to 0$$
in which $\gdim G_i =0$. By \cite{AuslanderBridger}, a local ring
is Gorenstein precisely when all its finitely generated modules
have finite Gorenstein dimension.

If $M$ has finite Gorenstein dimension $d$, say, then by
\cite[Corollary 3.15]{AuslanderBridger}, the module $\Omega_A^d(M)$
has Gorenstein dimension zero. Choose a minimal free resolution
$\mathbf{S} \to \Omega_A^d(M)^* \to 0$ of $\Omega_A^d(M)^*$, and
consider the dualized complex $0 \to \Omega_A^d(M) \to
\mathbf{S}^*$. It follows directly from the defining properties of
modules of Gorenstein dimension zero that this complex is exact.
Splicing this complex with the minimal free resolution of
$\Omega_A^d(M)$, we obtain a doubly infinite minimal exact sequence
$$\mathbf{Q} \colon \cdots \to Q_2 \xrightarrow{d_2} Q_1
\xrightarrow{d_1} Q_0 \xrightarrow{d_0} Q_{-1} \xrightarrow{d_{-1}}
Q_{-2} \to \cdots$$ of free modules, in which $\Im d_n =
\Omega_A^d(M)$. Then $\mathbf{Q}$ is a
\emph{minimal complete resolution} of $M$,
and it is unique up to homotopy equivalence (cf.\ \cite{Buchweitz},
\cite{CornickKropholler}). Consequently, for every $n \in
\mathbb{Z}$ and every $A$-module $N$, the \emph{stable homology} and
\emph{stable cohomology} modules
\begin{eqnarray*}
\Tatetor^A_n(M,N) & \stackrel{\text{def}}{=} & \H_n \left (
\mathbf{Q} \otimes_A N \right ) \\
\Tateext_A^n(M,N) & \stackrel{\text{def}}{=} & \H_{-n} \left (
\Hom_A( \mathbf{Q},N) \right )
\end{eqnarray*}
are independent of the choice of complete resolution of $M$. By
construction, there are isomorphisms $\Tatetor^A_n(M,N) \cong
\Tor_n^A(M,N)$ and $\Tateext_A^n(M,N) \cong \Ext_A^n(M,N)$ whenever
$n>d$.

By \cite[Theorem 1.4]{AvramovGasharovPeeva}, if the CI-dimension
of $M$ is finite, then
$$\gdim M = \cidim M = \depth A - \depth M.$$
Therefore $M$ admits a minimal complete resolution, and from the
above we see that, for every $A$-module $N$, there are isomorphisms
\begin{eqnarray*}
\Tatetor^A_n(M,N) & \cong & \Tor_n^A(M,N) \\
\Tateext_A^n(M,N) & \cong & \Ext_A^n(M,N)
\end{eqnarray*}
for all $n> \depth A - \depth M$. Consequently, for a module of
finite CI-dimension, vanishing patterns in stable (co)homology
correspond to vanishing patterns in ordinary (co)homology beyond
$\depth A - \depth M$. We shall therefore state the vanishing
results in terms of stable (co)homology.

\section{Vanishing of (co)homology}

In this section, we establish our rigidity results for stable
$\Ext$ and $\Tor$ for modules of finite CI-dimension. We start
with the following lemma, which shows that a module of finite
CI-dimension has reducible complexity.

\begin{lemma}\label{reducingcx}
Let $A$ be a local ring, and $M$ a finitely generated
$A$-module of finite CI-dimension and infinite projective dimension.
Then, given any odd number $q$, there exists a faithfully flat
extension $A \to R$ and an exact sequence
$$0 \to R \otimes_A M \to K \to \Omega_R^q(R \otimes_A M) \to 0$$
of $R$-modules, with $\cx_R K = \cx_A M -1$. Moreover, the
$R$-modules $R \otimes_A M$ and $K$ have finite CI-dimension, with
$\cidim_R (R \otimes_A M) = \cidim_R K = \depth A - \depth M$.
\end{lemma}

\begin{proof}
By \cite[Lemma 2.1]{Bergh2}, for any odd integer $q \ge 1$, there
exists a quasi-deformation $A \to R \leftarrow Q$ and an
exact sequence
$$0 \to R \otimes_A M \to K \to \Omega_R^{2q-1}(R \otimes_A M) \to 0$$
of $R$-modules, with $\cx_R K = \cx_A M -1$. Moreover, in the
proof of \cite[Lemma 2.1]{Bergh2} it is shown that the
CI-dimensions of both the $R$-modules $K$ and $R \otimes_A M$ are
finite. Since the CI-dimension of $R \otimes_A M$ is finite, so is
the CI-dimension of $\Omega_R^{2q-1}(R \otimes_A M)$, and by
\cite[Lemma 1.9]{AvramovGasharovPeeva} the inequality $\depth_R (R
\otimes_A M) \le \depth_R \Omega_R^{2q-1}(R \otimes_A M)$ holds.
But then $\depth_R K = \depth_R (R \otimes_A M)$, and so
$$\depth_R R - \depth_R K = \depth_R R - \depth_R (R \otimes_A M) = \depth_A A - \depth_A
M,$$ where the latter equality is due to faithful flatness.
\end{proof}

Having established the necessary lemma, we now prove the first of
the main results of this section.

\begin{theorem}\label{homologyvanishinggaps}
Let $A$ be a local ring, and $M$ a finitely generated
$A$-module of finite CI-dimension and complexity $c$. Furthermore,
let $N$ be a not necessarily finitely generated $A$-module. Suppose
there is an integer $n \in \mathbb{Z}$ and an odd number $q$ such
that
$$\Tatetor_n^A(M,N) = \Tatetor_{n+q}^A(M,N) = \cdots =
\Tatetor_{n+(c-1)q}^A(M,N) =0.$$ Then $\Tatetor_{n-i(q+1)}^A(M,N) =
\Tatetor_{n+(c-1)q+i(q+1)}^A(M,N)=0$ for all integers $i \ge 1$.
\end{theorem}

\begin{proof}
Denote $\depth A - \depth M$ by $d$. If $c=0$, then there is
nothing to prove since by the Auslander-Buchsbaum formula, the module
$\Omega_A^d(M)$ is free, and so $\Tatetor_i^A(M,N)=0$ for all $i$.

The proof proceeds by induction on the complexity $c$ of $M$. If
$c=1$, then by \cite[Theorem 7.3]{AvramovGasharovPeeva} the module
$\Omega_A^d(M)$ is periodic of period at most two, hence so is the
minimal complete resolution of $M$. In particular, the modules
$\Tatetor_i^A(M,N)$ and $\Tatetor_{i+2}^A(M,N)$ are isomorphic for
all integers $i$. Since $q$ is an odd number, the case $c=1$
follows.

Next, suppose that $c \ge 2$. Choose a faithfully flat extension
$A \to R$, together with an exact sequence
$$0 \to R \otimes_A M \to K \to \Omega_R^q(R \otimes_A M) \to 0$$
of $R$-modules, as in Lemma \ref{reducingcx}. Thus, the $R$-modules
$R \otimes_A M$ and $K$ have finite CI-dimension, and the complexity
of $K$ is $c-1$. For every $i \in \mathbb{Z}$ there is an
isomorphism $\Tatetor_i^R(R \otimes_A M, R \otimes_A N) \cong R
\otimes_A \Tatetor_i^A(M,N)$, hence $\Tatetor_i^A(M,N)$ vanishes if
and only if $\Tatetor_i^R(R \otimes_A M, R \otimes_A N)$ does. We
may therefore, without loss of generality, assume that there exists
an exact sequence
$$0 \to M \to K \to \Omega_A^q(M) \to 0$$
of $A$-modules, in which $K$ has finite CI-dimension and
complexity $c-1$. By the homology version of \cite[Proposition
5.6]{AvramovMartsinkovsky}, this short exact sequence induces a
doubly infinite long exact sequence
$$\cdots \to \Tatetor_{i+1}^A(K,N) \to \Tatetor_{i+1}^A( \Omega_A^q(M),N) \to
\Tatetor_i^A(M,N) \to \Tatetor_i^A(K,N) \to \cdots$$ of complete
homology modules. Using \cite[Proposition
5.6]{AvramovMartsinkovsky} once more, together with the fact that
$\Tatetor_i^A(F,N)=0$ for all $i$ whenever $F$ is free, we see
that $\Tatetor_i^A( \Omega_A^q(M),N)$ is isomorphic to
$\Tatetor_{i+q}^A(M,N)$ for all $i$. Consequently, we obtain a
long exact sequence
$$\cdots \to \Tatetor_{i+1}^A(K,N) \to \Tatetor_{i+q+1}^A(M,N) \to
\Tatetor_i^A(M,N) \to \Tatetor_i^A(K,N) \to \cdots$$ of complete
homology modules.

The vanishing assumption on $\Tatetor_i^A(M,N)$ forces
$\Tatetor_i^A(K,N)$ to vanish for $i \in \{ n, n+q, \dots,
n+(c-2)q \}$. By induction, the modules $\Tatetor_{n-i(q+1)}^A(K,N)$
and $\Tatetor_{n+(c-2)q+i(q+1)}^A(K,N)$ vanish for all integers $i
\ge 1$. Looking at the above long exact sequence again, we see
that for all integers
$i \ge 1$ the modules $\Tatetor_{n-i(q+1)}^A(M,N)$ and
$\Tatetor_{n+(c-1)q+i(q+1)}^A(M,N)$ also must vanish.
\end{proof}

We include the cohomology version of Theorem
\ref{homologyvanishinggaps}, but omit the proof.

\begin{theorem}\label{cohomologyvanishinggaps}
Let $A$ be a local ring, and $M$ a finitely generated
$A$-module of finite CI-dimension and complexity $c$. Furthermore,
let $N$ be a not necessarily finitely generated $A$-module.
Suppose there is an integer $n \in \mathbb{Z}$ and an odd number
$q$ such that
$$\Tateext^n_A(M,N) = \Tateext^{n+q}_A(M,N) = \cdots =
\Tateext^{n+(c-1)q}_A(M,N) =0.$$ Then $\Tateext^{n-i(q+1)}_A(M,N) =
\Tateext^{n+(c-1)q+i(q+1)}_A(M,N)=0$ for all integers $i \ge 1$.
\end{theorem}

In the following corollaries, we record the special case $q=1$
from the previous theorems.

\begin{corollary}\label{homologgapone}
Let $A$ be a local ring, and $M$ a finitely generated
$A$-module of finite CI-dimension and complexity $c$. Furthermore,
let $N$ be a not necessarily finitely generated $A$-module.
Suppose there is an integer $n \in \mathbb{Z}$ such that
$$\Tatetor_n^A(M,N) = \Tatetor_{n+1}^A(M,N) = \cdots =
\Tatetor_{n+c-1}^A(M,N) =0.$$ Then $\Tatetor_{n-2i}^A(M,N) =
\Tatetor_{n+c-1+2i}^A(M,N)=0$ for all integers $i \ge 1$.
\end{corollary}

\begin{corollary}\label{cohomologgapone}
Let $A$ be a local ring, and $M$ a finitely generated
$A$-module of finite CI-dimension and complexity $c$. Furthermore,
let $N$ be a not necessarily finitely generated $A$-module.
Suppose there is an integer $n \in \mathbb{Z}$ such that
$$\Tateext_n^A(M,N) = \Tateext_{n+1}^A(M,N) = \cdots =
\Tateext_{n+c-1}^A(M,N) =0.$$ Then $\Tateext_{n-2i}^A(M,N) =
\Tateext_{n+c-1+2i}^A(M,N)=0$ for all integers $i \ge 1$.
\end{corollary}

We note that Theorems \ref{homologyvanishinggaps} and
\ref{cohomologyvanishinggaps} recover results of \cite{Jorgensen2}
and \cite{Bergh2} for the vanishing of $\cx_AM+1$ consecutive
$\Ext$ and $\Tor$ for modules of finite CI-dimension.

\begin{corollary} Let $A$ be a local ring, and $M$ a finitely generated
$A$-module of finite CI-dimension and complexity $c$. Furthermore,
let $N$ be a not necessarily finitely generated $A$-module.
Suppose there is an integer $n\ge \depth A -\depth M +1$ such that
$$\Tor_n^A(M,N) = \Tor_{n+1}^A(M,N) = \cdots =
\Tor_{n+c}^A(M,N) =0.$$ Then $\Tor_i^A(M,N) = 0$
for all integers $i \ge \depth A -\depth M +1$.
\end{corollary}

\begin{corollary} Let $A$ be a local ring, and $M$ a finitely generated
$A$-module of finite CI-dimension and complexity $c$. Furthermore,
let $N$ be a not necessarily finitely generated $A$-module.
Suppose there is an integer $n\ge \depth A -\depth M +1$ such that
$$\Ext_n^A(M,N) = \Ext_{n+1}^A(M,N) = \cdots =
\Ext_{n+c}^A(M,N) =0.$$ Then $\Ext_i^A(M,N) = 0$
for all integers $i \ge \depth A -\depth M +1$.
\end{corollary}

We also generalize a result of \cite{Jorgensen} for vanishing of
Tor for modules over complete intersections.

\begin{theorem}\label{torci}
Let $A$ be a local Cohen-Macaulay ring of dimension $d$, and $M$
and $N$ finitely generated $A$-modules with $M$ of finite
CI-dimension. Then there exists an integer $n_0$ with the following
property: if
$$\Tatetor_i^A(M,N) = \Tatetor_{i+1}^A(M,N) = \cdots = \Tatetor_{i+d}^A(M,N)=0$$
for one even $i \ge n_0$, and
$$\Tatetor_j^A(M,N) = \Tatetor_{j+1}^A(M,N) = \cdots = \Tatetor_{j+d}^A(M,N)=0$$
for one odd $j \ge n_0$, then $\Tatetor_n^A(M,N)=0$ for all $n \in
\mathbb{Z}$.
\end{theorem}

\begin{remark} Using the fact that for finitely generated $A$-modules
$M$ and $N$ with $M$ maximal Cohen-Macaulay,
$$
\Tatetor^R_i(M,N)\cong\Tateext_R^{-i-1}(M^*,N)
$$
for all $i\in\mathbb Z$, one has a statement similar to that of \ref{torci} for
vanishing of stable Ext with $i\le n_0$ and $j\le n_0$.
\end{remark}

\begin{proof}
We prove this result in terms of vanishing of the ordinary
homology modules $\Tor_n^A(M,N)$ for $n> \depth A - \depth M$.
For, by \cite[Theorem 4.9]{AvramovBuchweitz}, the complete
homology modules $\Tatetor_n^A(M,N)$ vanish for all $n \in
\mathbb{Z}$ if and only if $\Tor_n^A(M,N)=0$ for $n> \depth A -
\depth M$.

The proof is by induction on $d$, the case $d=0$ being covered by
\cite[Theorem 3.1]{Jorgensen} (strictly speaking, the result
\cite[Theorem 3.1]{Jorgensen} is formulated for modules over
complete intersections, but the proof carries over verbatim to
modules of finite CI-dimension). Suppose therefore that $d$ is
positive. We may assume that both $M$ and $N$ are of positive
depth; if not, then we replace them by their first syzygies
$\Omega_A^1(M)$ and $\Omega_A^1(N)$. By \cite[Lemma
1.9]{AvramovGasharovPeeva}, the module $\Omega_A^1(M)$ also has
finite CI-dimension.

Choose an element $x \in A$ which is regular on $M,N$ and $A$, and
consider the exact sequence
$$0 \to M \xrightarrow{\cdot x} M \to M/xM \to 0.$$
This sequence induces a long exact sequence
$$\cdots \to \Tor_i^A(M,N) \xrightarrow{\cdot x} \Tor_i^A(M,N) \to \Tor_i^A(M/xM,N) \to
\Tor_{i-1}^A(M,N) \to \cdots$$ in homology. Now denote the ring
$A/(x)$ by $\bar{A}$, and the $\bar{A}$-modules $M/xM$ and $N/xN$ by
$\bar{M}$ and $\bar{N}$, respectively. Note that, by
\cite[Proposition 1.12]{AvramovGasharovPeeva}, the $\bar{A}$-module
$\bar{M}$ has finite CI-dimension. Thus, since the dimension of
$\bar{A}$ is $d-1$, by induction there exists an integer $n_0$ with
the following property: if
$$\Tor_i^{\bar{A}}( \bar{M}, \bar{N} ) = \Tor_{i+1}^{\bar{A}}( \bar{M}, \bar{N} ) = \cdots =
\Tor_{i+d-1}^{\bar{A}}( \bar{M}, \bar{N} )=0$$ for one even $i \ge
n_0$, and
$$\Tor_j^{\bar{A}}( \bar{M}, \bar{N} ) = \Tor_{j+1}^{\bar{A}}( \bar{M}, \bar{N} ) = \cdots =
\Tor_{j+d-1}^{\bar{A}}( \bar{M}, \bar{N} )=0$$ for one odd $j \ge
n_0$, then $\Tor_n^{\bar{A}}( \bar{M}, \bar{N} )=0$ for all $n> \dim
\bar{A} - \depth \bar{M}$. Note that $\dim \bar{A} - \depth \bar{M}
= d - \depth M$.

Suppose
$$\Tor_i^A(M,N) = \Tor_{i+1}^A(M,N) = \cdots = \Tor_{i+d}^A(M,N)=0$$
for one even $i \ge n_0-1$, and
$$\Tor_j^A(M,N) = \Tor_{j+1}^A(M,N) = \cdots = \Tor_{j+d}^A(M,N)=0$$
for one odd $j \ge n_0-1$. Then the above long exact homology
sequence implies that $\Tor_n^A( \bar{M} ,N)=0$ for $i+1 \le n \le
i+d$ and $j+1 \le n \le j+d$. By \cite[Lemma 18.2(iii)]{Matsumura},
there is an isomorphism $\Tor_n^A( \bar{M},N) \cong
\Tor_n^{\bar{A}}( \bar{M}, \bar{N} )$ for every $n>0$, and so from
above we see that $\Tor_n^A( \bar{M},N)$ vanishes for all $n> d -
\depth M$. The long exact homology sequence then shows that
$\Tor_n^A(M,N)= x \Tor_n^A(M,N)$ for all $n> d - \depth M$, and by
Nakayama's Lemma we conclude that $\Tor_n^A(M,N)=0$ for all $n> d -
\depth M$.
\end{proof}

\begin{corollary} Let $A$ be a local Cohen-Macaulay ring of
dimension $d$, and $M$ and $N$ finitely generated $A$-modules
with $M$ of finite CI-dimension.  If for all positive integers
$n$ there exists an $i\ge n$ such that
\[
\Tatetor_i^A(M,N) = \Tatetor_{i+1}^A(M,N) = \cdots = \Tatetor_{i+d+1}^A(M,N)=0
\]
then $\Tatetor_n^A(M,N) = 0$ for all $n\in\mathbb Z$.

\end{corollary}
We remark that the examples of \cite[4.1]{Jorgensen} illustrate
the sharpness of Theorems \ref{homologyvanishinggaps} and
\ref{cohomologyvanishinggaps} in the $q=1$ case, in the sense that
more vanishing cannot be concluded from the hypothesis.  We recall
these examples, in the context of stable (co)homology, and prove
that certain homology modules remain nonzero.

\begin{example} Let $n$ be a positive integer and
\[
R=k[[X_1,\dots,X_n,Y_1,\dots,Y_n]]/(X_1Y_1,\dots,X_nY_n),
\]
where $k$ is a field and the $X_i$ and $Y_i$ are analytic indeterminates.
Then $R$ is a complete intersection of dimension $n$ and codimension $n$.
Let $M=R/(x_1,\dots,x_n)$, and $N=R/(y_1,\dots,y_n)$. Then, as is shown
in \cite{Jorgensen}, $M$ and $N$ are maximal Cohen-Macaulay $R$-modules
of complexity $n$ with $\Tateext_R^i(M,N)=0$ for $0\le i\le n-1$, and
$\Tateext_R^{-1}(M,N)\ne 0\ne \Tateext_R^n(M,N)$.  Theorem
\ref{cohomologyvanishinggaps} shows that $\Tateext_R^{-2i}(M,N)=0$ for
all $i \ge 1$.  We moreover claim that $\Tateext_R^{-2i-1}(M,N)\ne 0$
for all $i \ge 1$.

Indeed, note that $M\cong M^*$.
Consider the ring $S=k[[X_1,Y_1]]/(X_1Y_1)$, and $S$-module
$M'=S/(x_1)$.  One can construct a chain map $f$ between the minimal
resolution $\cdots \to S \xrightarrow{x_1} S \xrightarrow{y_1} S
\xrightarrow{x_1} S\to M'\to 0$  of $M'$ over $S$ and one $F$ of
$M$ over $R$ (and consequently one of $M^*$ as well)
\[
\xymatrixrowsep{2pc} \xymatrixcolsep{2pc} \xymatrix{
\cdots \ar@{->}[r] & S \ar@{->}^{x_1}[r]\ar@{->}^{f_3}[d] & S \ar@{->}^{y_1}[r]\ar@{->}^{f_2}[d] &
S \ar@{->}^{f_1}[d]\ar@{->}^{x_1}[r] & S \ar@{->}[r]\ar@{->}^{f_0}[d] & M' \ar@{->}[r]\ar@{->}^{\epsilon}[d] & 0\\
\cdots \ar@{->}[r] & F_3 \ar@{->}^{\partial_3}[r] & F_2 \ar@{->}^{\partial_2}[r] &
F_1 \ar@{->}^{\partial_1}[r] & R \ar@{->}[r] & M \ar@{->}[r] & 0
}
\]
such that $f_i(1)$ is a basis element of $F_i$ for all $i\ge 0$.
Tensoring the top row with $N'=S/(y_1)$ and the bottom row with
$N$, we get an induced commutative diagram
\[
\xymatrixrowsep{2pc} \xymatrixcolsep{2.5pc} \xymatrix{
\cdots \ar@{->}[r] & N' \ar@{->}^{\bar x_1}[r]\ar@{->}^{\bar f_3}[d] & N' \ar@{->}^{0}[r]\ar@{->}^{\bar f_2}[d] & N' \ar@{->}^{\bar f_1}[d]\ar@{->}^{\bar x_1}[r] & N' \ar@{->}[r]\ar@{->}^{\bar f_0}[d] & 0\\
\cdots \ar@{->}[r] & F_3\otimes_R N \ar@{->}^{\partial_3\otimes N}[r] & F_2\otimes_R N \ar@{->}^{\partial_2\otimes N}[r] &
F_1\otimes_R N \ar@{->}^{\partial_1\otimes N}[r] & N \ar@{->}[r] & 0
}
\]
in which $\bar f_i(\bar 1)$ is a minimal generator of $F_i\otimes_R N$
for all $i\ge 0$.  It follows that $\Tor^R_{2i}(M,N)\ne 0$ for all
$i\ge 0$.  Finally we note that a complete resolution of $M\cong M^*$ is given by
\[
\xymatrixrowsep{.1pc} \xymatrixcolsep{2pc} \xymatrix{
\cdots \ar@{->}[r] & F_2 \ar@{->}^{\partial_2}[r] & F_1 \ar@{->}^{\partial_1}[r] &
F_0 \ar@{->}^{[y_1\cdots y_2]}[rr] & & F_0 \ar@{->}[r]^{\partial^*_1} & F_1 \ar@{->}[r]^{\partial^*_2} &
F_2 \ar@{->}[r] & \cdots\\
&&&& M \ar@{<-}[ul]\ar@{->}[ur]&&&&\\
&&&0 \ar@{->}[ur] && 0 \ar@{<-}[ul]
}
\]
and so $\Tatetor^R_i(M,N)=\Tateext_R^{-i-1}(M,N)$ for all $i\in\mathbb Z$.  Thus
$\Tateext^{-2i-1}_R(M,N)\ne 0$ for all $i\ge 1$, and
this is what we claimed.
\end{example}

\section{Pre-rigidity of Modules}

Throughout this section, unless otherwise specified
we let $(Q,\n,k)$ be a local ring, $x$ a
non-zerodivisor contained in the maximal ideal of $Q$, and
$R=Q/(x)$. Let $M$ be a finitely generated non-zero $R$-module, and
$F$ a $Q$-free resolution of $M$. Assume that
$\{\sigma_i\}_{i\ge 0}$ is a system of higher homotopies on $F$.
That is, for all $i\ge 0$ each $\sigma_i$ is a degree $2i-1$
endomorphisms of $F$ as a graded module with
$\sigma_0=\partial^F$, $\sigma_0\sigma_1+\sigma_1\sigma_0=x\Id_F$
and $\sum_{i+j=n}\sigma_i\sigma_j=0$ for $n>1$. (Shamash shows in
\cite{Shamash} that such a system always exists.)

\begin{definition}
We say that an $R$-module $N$ is
{\em pre-rigid of degree $r$ with respect to $M$ and $Q$}
if there exists a $Q$-free resolution $F$ of $M$ and a system of
higher homotopies $\{\sigma_i\}_{i\ge 0}$ on $F$ such that the
induced maps
\[
(\sigma_i)_j\otimes_Q N: F_j\otimes_Q N \to F_{j+2i-1}\otimes_Q N
\]
are zero for $j>r-(2i-1)$, and all $i\ge 1$.
\end{definition}

\begin{example}
If $\pd_QM = r<\infty$, then every $R$-module
$N$ is pre-rigid of degree $r$ with respect to $M$ and $Q$.
\end{example}

\begin{example}
Suppose that $\sigma_i(F)\subseteq\n F$ for all $i\ge 1$.
Then $k$ is pre-rigid of degree $0$ with respect to $M$ and $Q$.
%(If $\sigma_0(F)\subseteq\n F$, so that $F$ is a minimal
%$Q$-free resolution of $M$, then $\sigma_i(F)\subseteq\n F$ for all
%$i\ge 1$ precisely
%when $D\otimes_Q F$ is a minimal $R$-free resolution of $M$.)
\end{example}

The following is the main result of this section. It motivates the
choice of terminology.

\begin{theorem}\label{premain} Let $M$ be a finitely generated $R$-module,
and assume that $N$ is an $R$-module which is pre-rigid of degree $r$
with respect to $M$ and $Q$. If $\Tor_n^R(M,N)=0$ for some $n>r$,
then $\Tor^Q_{n-2i}(M,N)=0$ for $n\ge n-2i > r$.
If $r=0$, then $\Tor^Q_{n-2i}(M,N)=0$ for all $i\ge 0$.
\end{theorem}

In preparation for the proof of Theorem \ref{premain}
we want to describe a free resolution of $M$ over $R$
using one of $M$ over $Q$, following
\cite{Shamash} (see also \cite[3.1.3]{Avramov}).

Let $D$ be the complex of $R$-modules with trivial differential
having $D_i=0$ for $i<0$, $D_{2i-1}=0$ for $i\ge 1$, and $D_{2i}$ the free
$R$-module $Re_i$ on the singleton basis $e_i$ for $i\ge 0$.
Let $F$ be a free resolution of $M$ over $Q$, and $\{\sigma_i\}_{i\ge 0}$
a system of higher homotopies on $F$ (recall that $\sigma_0$ is the
differential of $F$). We equip the complex
$D \otimes_Q F$ with the differential $\partial=\sum_jt^j\otimes\sigma_j$
where $t^j$ is defined by $t^j(e_i)=e_{i-j}$, so that
$\partial(e_i\otimes f)=\sum_je_{i-j}\otimes \sigma_j(f)$. Then
$(D\otimes_Q F,\partial)$ is a free resolution of $M$ over $R$
\cite{Shamash}.

\begin{proof}
We may compute $\Tor^R_i(M,N)$ from the complex
\[
\mathcal F = (D \otimes_Q F) \otimes_R N\cong D \otimes_Q F\otimes_Q N.
\]
Filtering this complex by
$\mathcal F_p = \sum_{i \le p} D_{2i} \otimes_Q F \otimes_Q N$ one gets
an upper semi-first-quadrant convergent spectral sequence whose $E^0$-page is
\[
\xymatrixrowsep{2pc} \xymatrixcolsep{2.9pc} \xymatrix{ \vdots
\ar@{->}[d] & \vdots \ar@{->}[d] &\vdots \ar@{->}[d] &\vdots
\ar@{->}[d] \\
D_0 \otimes F_3 \otimes N \ar@{->}[d] & D_2 \otimes F_2 \otimes N \ar@{->}[d] \ar@{->}[l]& D_4 \otimes F_1 \otimes N\ar@{->}[d] \ar@{->}[l] \ar@{->}[ull]_{t^2\otimes \sigma_2\otimes N} & D_6 \otimes F_0 \otimes N \ar@{->}[l] \ar@{->}[ull]\\
D_0 \otimes F_2 \otimes N\ar@{->}[d] & D_2 \otimes F_1 \otimes N\ar@{->}[d] \ar@{->}[l]& D_4 \otimes F_0 \otimes N \ar@{->}[l]\ar@{->}[ull]& {}\\
D_0 \otimes F_1 \otimes N\ar@{->}[d]^{D_0\otimes\sigma_0\otimes N} & D_2 \otimes F_0 \otimes N \ar@{->}[l]_{t\otimes \sigma_1\otimes N} & {} & {} \\
D_0 \otimes F_0 \otimes N& {} & {} \\
}
\]
with the convention that $E^0_{i,j} = D_{2i}\otimes_Q F_{j-i}$. Since
$D_{2i}\cong R$ for all $i\ge 0$, the $E^1$-page of this spectral sequence is
\[
\xymatrixrowsep{2pc} \xymatrixcolsep{2pc} \xymatrix{\vdots & \vdots
&\vdots &\vdots \\
\Tor_3^Q(M,N) &  \Tor_2^Q(M,N) \ar@{->}[l]_{d_{1,3}^1}& \Tor_1^Q(M,N)
\ar@{->}[l]_{d_{2,3}^1}& \Tor_0^Q(M,N) \ar@{->}[l]_{d_{3,3}^1} \\
\Tor_2^Q(M,N) &  \Tor_1^Q(M,N) \ar@{->}[l]_{d_{1,2}^1}&
\Tor_0^Q(M,N) \ar@{->}[l]_{d_{2,2}^1} & {} \\
\Tor_1^Q(M,N) & \Tor_0^Q(M,N) \ar@{->}[l]_{d^1_{1,1}} & {} & {}\\
\Tor_0^Q(M,N) & {} & {} & {} }
\]
where the maps $d_{1,i}^1$ are induced by the maps
\[
t\otimes (\sigma_1)_{i-1}\otimes N :
D_2\otimes F_{i-1} \otimes N \to D_0\otimes F_i\otimes N
\]
for $i\ge 1$. Note that $d_{i,j}^1=d_{i+1,j+1}^1$ for all $i,j\ge 1$.

Now assume that $N$ is pre-rigid of degree $r$ with respect to $M$ and $Q$.
Then it is clear that
the maps $d^1_{ 1,j}=0$ for all $j\ge r$, and thus $d^1_{i,j}=0$ for
all $j\ge i+r$.
It follows that $E^2_{i,j}=E^1_{i,j}=\Tor^Q_{j-i}(M,N)$ for all
$j\ge i+r+1$.  In general, the hypothesis
that $N$ is pre-rigid implies that the maps $d^s_{i,j}$ on the
$E^s$-page of the spectral sequence are zero for
all $j\ge i+r-(2s-1)+1$, and all $s\ge 1$, and thus the
limit terms of the spectral sequence are given by
\[
E^\infty_{i,j}=E^1_{i,j}=\Tor^Q_{j-i}(M,N)
\]
for all $j\ge i+r+1$.

Now taking the associated filtration $\Phi$ of the total homology
$H$ of $\mathcal F$ (see, for example, \cite[11.13]{Rotman}), we have
isomorphisms $\Tor^Q_{j-i}(M,N)\cong\Phi^iH_{i+j}/\Phi^{i-1}H_{i+j}$
for $j\ge i+r+1$.
Since $H_n=\Tor^R_n(M,N)$ for all $n$, the first statement of
Theorem \ref{premain} follows easily.

When $r=0$ we actually get that
$E^\infty_{i,j}=E^1_{i,j}=\Tor^Q_{j-i}(M,N)$
for all $j\ge i$, and so the second statement of the theorem holds.
\end{proof}

The following main corollary of \ref{premain} shows that
the notion of pre-rigidity generalizes in a sense
the behavior of modules of finite CI-dimension and
complexity one.

\begin{corollary} Let $A$ be a local ring, and
assume that $M$ is a finitely $A$-module with
finite CI-dimension.  Let
$A\to R\leftarrow Q$ be a codimension $c$ quasi-deformation with
$R\cong Q/(x_1,\dots,x_c)$ such that
$\pd_Q M\otimes_AR<\infty$.
Assume that $N$ is an $A$-module such that
$N\otimes_AR$ is pre-rigid of degree $r$
with respect to $M\otimes_AR$ and $Q/(x_2,\dots,x_c)$.
Set $b=\max\{r,\depth A-\depth_AM+1\}+c$.
If $\Tor^A_n(M,N)=0$ for one even value of
$n\ge b$ and one odd value of $n\ge b$,
then $\Tor^A_n(M,N)=0$ for all $n\ge\depth A-\depth M+1$.
\end{corollary}

\begin{proof}
Suppose that $\Tor^A_{n_e}(M,N)=0$ for an even $n_e\ge b$
and $\Tor^A_{n_o}(M,N)=0$ for an odd $n_o\ge b$.  By flatness we have
$\Tor^R_{n_e}(M',N')=\Tor^R_{n_o}(M',N')=0$, where $M'=R\otimes_A M$
and $N'=R\otimes_A N$.  Let $Q'=Q/(x_2,\dots,x_c)$.  By assumption
$N'$ is pre-rigid of degree $r$ with respect to $M'$ and $Q'$.
Since $b> r$, Theorem \ref{premain} applies to give
$\Tor ^{Q'}_{n-j}(M',N')=0$ for $n\ge n-j>r$, where
$n=\min\{n_o,n_e\}$. Since $n-r\ge b-r\ge c$, and
$n-(\depth A-\depth_AM+1)\ge b-(\depth A-\depth_AM+1)\ge c$,
We have at least $c$ consecutive vanishing $\Tor ^{Q'}_{n-j}(M',N')=0$
beyond $\depth A-\depth_AM+1=\depth Q'-\depth_{Q'}M'$.
The complexity of $M'$ as a $Q'$-module is at most $c-1$. Thus by
\cite[2.2]{Jorgensen2} we have $\Tor^{Q'}_j(M',N')=0$ for all
$j\ge\depth Q'-\depth_{Q'}M'+1$.  A standard argument (see, for example,
\cite[0.1]{Jorgensen}) now shows that
$\Tor^R_j(M',N')\cong\Tor^R_{j+2}(M',N')$ for all
$j\ge \depth R -\depth M'+1$.  Finally, since
$\Tor^R_{n_e}(M',N')=\Tor^R_{n_o}(M',N')=0$ it follows that
$\Tor^R_j(M',N')=0$ for all $j\ge \depth R -\depth_RM'+1$.
Thus $\Tor^A_j(M,N)=0$ for all $j\ge\depth A-\depth_AM +1$,
which was the claim.
\end{proof}

The next corollary is an immediate consequence of Theorem \ref{premain}.

\begin{corollary} Let $M$ be a finitely generated non-zero $R$-module.
Suppose that $N$ is an $R$-module which is pre-rigid of degree $0$
with respect to $M$ and $Q$.
Then $\Tor^R_n(M,N)=0$ for some even $n\ge 0$ if and only if $N=0$.
\end{corollary}

The next theorem shows that the pre-rigidity condition gives
a formula for relative lengths of Tor.

\begin{theorem}\label{presecond} Let $M$ be a finitely generated
$R$-module.
Suppose that $N$ is an $R$-module which is pre-rigid of degree $0$
with respect to $M$ and $Q$.
If $\Tor^R_n(M,N)$ has finite length for some $n\ge 0$, then
$\Tor^Q_{n-2i}(M,N)$ has finite
length for all $i\ge 0$, and
\[
\length\Tor^R_n(M,N)=\sum_{i\ge 0}\length\Tor^Q_{n-2i}(M,N)
\]
\end{theorem}

\begin{proof}
Consider the spectral sequence in the proof of Theorem \ref{premain}.
The associated filtration $\Phi$ of the total homology
$H$ of $\mathcal F$ is
\[
0=\Phi^{-1}H_n\subseteq\Phi^{0}H_n\subseteq\cdots
\subseteq\Phi^{n-1}H_n\subseteq\Phi^{n}H_n=H_n
\]
for all $n$,
and we have $E^\infty_{i,j}\cong \Phi^{i}H_n/\Phi^{i-1}H_n$
for $i+j=n$, and all $n$. If $N$ is pre-rigid of degree 0
with respect to $M$ and $Q$, then as we saw in the proof of
Theorem \ref{premain}, $E^\infty_{i,j}\cong \Tor^Q_{j-i}(M,N)$
for all $i,j$.  Recalling that $H_n\cong\Tor^R_n(M,N)$,
the claim is now clear.
\end{proof}

We single out a particular case of interest, which follows
directly from Theorem \ref{presecond}. (See, for example,
the definition of $\theta^R(M,N)$ in \cite{Dao1} and \cite{Dao2}.)

\begin{corollary}  Let $M$ be a finitely generated $R$-module.
Suppose that $N$ is an $R$-module which is pre-rigid of degree $0$
with respect to $M$ and $Q$.
Then $\Tor_i^R(M,N)$ has finite length for some even $i\ge 0$
if and only if
$M\otimes_RN$ has finite length.
\end{corollary}

\begin{remark} Suppose that a free resolution $F$ of $M$ over $Q$
admits a system of higher homotopies $\{\sigma_i\}_{i\ge 0}$
such that $\sigma_i(F)\subseteq \n F$ for all $i\ge 0$.  Then
the $R$-free resolution $D\otimes F$ of $M$ in the proof of
Theorem \ref{premain} will be minimal.  In this
case we see that $k$ is pre-rigid of degree 0 with respect to $M$ and $F$.
Theorem \ref{presecond} then gives a statement about Betti numbers:
$\beta_n^R(M)=\sum_{i\ge 0}\beta_{n-2i}^Q(M)$, which in terms
of Poincar\'e series translates to $P^R_M(t)=P^Q_M(t)/(1-t^2)$.
Indeed, this was the main goal of \cite{Shamash}.  It is shown in {\em loc. cit.}
that the condition on the minimality of $D\otimes_Q F$ is obtained, for example,
when $x\in \n \Ann_Q M$.
\end{remark}

\section*{Acknowledgements}

This work was done while the second author was visiting Trondheim,
Norway, December-January 2008-9.  He thanks the Algebra Group at
the Institutt for Matematiske Fag, NTNU, for their hospitality and
generous support. The first author was supported by NFR Storforsk
grant no.\ 167130.

\end{document}